\documentclass{amsart}
\usepackage{amssymb}
\newtheorem{theorem}{Theorem}[section]

\newtheorem{conjecture}{Conjecture}[section]
\theoremstyle{definition}

\theoremstyle{remark}

\numberwithin{equation}{section}

\begin{document}

\title[APPROXIMATIONS FOR THE GAMMA FUNCTION]{MORE ACCURATE APPROXIMATIONS\\ FOR THE GAMMA FUNCTION}
\author{Gerg\H{o} Nemes}	
\address{Lor\'and E\"otv\"os University, H-1117 Budapest, P\'azm\'any P\'eter s\'et\'any 1/C, Hungary}
\email{nemesgery@gmail.com}


\begin{abstract}
A series transformation idea inspired by a formula of R. W. Gosper and some asymptotic expansions for the central binomial coefficients leads us to new accurate approximations for the Gamma function.
\end{abstract}

\maketitle

\section{Introduction}
The Gamma function plays an important role in several fields of mathematics such as probability theory or combinatorics. One often has to evaluate the function for large positive values. One way to aim this is to use asymptotic approximations. It is well known that for large values of $x$ the Gamma function has the asymptotic series of the form \cite{ref3,ref5,ref4,ref7}
\begin{equation}\label{formula1}
\Gamma \left( x+1 \right) \sim x^x e^{ - x} \sqrt {2\pi x} \left( {1 + \frac{1}{{12x}} + \frac{1}{{288x^2 }} - \frac{{139}}{{51840x^3 }} - \frac{{571}}{{2488320x^4 }} +  \cdots } \right).
\end{equation}
Equation \eqref{formula1} is called Stirling's formula however, Laplace was the first who derived it by his approximation method for special integrals. Another famous result is the Stirling series \cite{ref3,ref5,ref7}
\begin{equation}\label{formula2}
\log \Gamma \left( {x + 1} \right) \sim \left( {x + \frac{1}{2}} \right)\log x -x +\frac{1}{2} \log\left({2\pi }\right)  + \frac{1}{{12x}} - \frac{1}{{360x^3 }} + \frac{1}{{1260x^5 }} -  \cdots .
\end{equation}
A main advantage of this latter series is that it has only odd powers of the variable.  For the past almost three hundred years several authors established fascinating new asymptotic formulas to improve the accuracy of \eqref{formula1}. For example, Karatsuba showed that a formula of Ramanujan can turn into an asymptotic expansion \cite{ref2}:
\begin{equation}
\Gamma \left( {x + 1} \right) \sim x^x e^{ - x} \sqrt \pi  \sqrt[6]{{8x^3  + 4x^2  + x + \frac{1}{{30}} - \frac{{11}}{{240x}} +  \cdots }}.
\end{equation}
Mortici proved in his more recent paper \cite{ref1} the following expansion similar to Karatsuba's:
\begin{equation}
n! = \Gamma \left( {n + 1} \right) \sim n^n e^{ - n} \sqrt \pi  \sqrt {2n + \frac{1}{3} + \frac{1}{{36n}} - \frac{{31}}{{3240n^2 }} - \frac{{139}}{{77760n^3 }} +  \cdots } .
\end{equation}
We develop some new variants of \eqref{formula1} in this paper  and show that these formulas are numerically more efficient than much of the early ones in many cases. The first few values of the newly introduced coefficients and sequences can be found in Appendix \ref{appendix}.

\section{New asymptotic expansions}

The motivating examples are the following two asymptotic expansions \cite{ref7}:
\begin{equation}
\binom{2n}{n} \sim \frac{{2^{2n} }}{{\sqrt {\pi n} }}\left( {1 - \frac{1}{{8n}} + \frac{1}{{128n^2 }} + \frac{5}{{1024n^3 }} - \frac{{21}}{{32768n^4 }} +  \cdots } \right),
\end{equation}
\begin{equation}\label{formula3}
\binom{2n}{n} \sim \frac{{2^{2n} }}{{\sqrt {\pi \left( {n + \frac{1}{4}} \right)} }}\left( {1 - \frac{1}{{64\left( {n + \frac{1}{4}} \right)^2 }} + \frac{{21}}{{8192\left( {n + \frac{1}{4}} \right)^4 }} - \cdots } \right).
\end{equation}
The first one is the standard asymptotic series of the central binomial coefficients. If one expands them into a series in powers of $1/\left(n+1/4\right)$, the asymptotic series contains only even powers. This remarkable result suggests that there might have been an asymptotic expansion similar to \eqref{formula1} that is, it contains only even powers of the shifted variable. The formula
\begin{equation}
\Gamma \left( {x + 1} \right) = x^x e^{ - x} \sqrt {2\pi \left( {x + \frac{1}{6}} \right)} \left( {1 + \mathcal{O}\left( {\frac{1}{{x^2 }}} \right)} \right),
\end{equation}
known as Gosper's approximation \cite{ref6} can be a good starting point. Our aim is to elaborate the asymptotic series part in Gosper's formula. It seems from \eqref{formula3} that another series in terms of $1/\left(x+1/6\right)$ would be the right choice. It can be shown that a series like that contains even and odd powers. If we insist to have even powers only we are lead to the form
\begin{equation}
\Gamma \left( {x + 1} \right) \sim x^x e^{ - x} \sqrt {2\pi \left( {x + \frac{1}{6}} \right)} \left( {\frac{{g_0 }}{{\left( {x + v_0 } \right)^0 }} + \frac{{g_1 }}{{\left( {x + v_1 } \right)^2 }} + \frac{{g_2 }}{{\left( {x + v_2 } \right)^4 }} +  \cdots } \right),
\end{equation}
where the sequences $\left\{ {g_n } \right\}_{n \geq 0}$ and $\left\{ {v_n } \right\}_{n \geq 0}$ has to be determined. One of our main result is

\begin{theorem} The Gamma function has an asymptotic series expansion of the form
\begin{equation}\label{expansion}
\Gamma \left( {x + 1} \right) \sim x^x e^{ - x} \sqrt {2\pi \left( {x + \frac{1}{6}} \right)}  \sum\limits_{n \ge 0} {\frac{{g_n }}{{\left( {x + v_n } \right)^{2n} }}} ,
\end{equation}
as $x \rightarrow \infty$, where the sequences $\left\{ {g_n } \right\}_{n \geq 0}$ and $\left\{ {v_n } \right\}_{n \geq 0}$ can be found from the recurrence
\begin{equation}\label{formula4}
\sum\limits_{j = 0}^n {\binom{ - 1/2}{j} \frac{{a_{n - j} }}{{6^j }}}  = \sum\limits_{j = 0}^{\left\lfloor {n/2} \right\rfloor } {\binom{ - 2j}{n - 2j}g_j v_j^{n - 2j} } .
\end{equation}
Here the $a_n$ coefficients are those appearing in \eqref{formula1}, i.e.,
\begin{equation}
\frac{{\Gamma \left( {x + 1} \right)}}{{x^x e^{ - x} \sqrt {2\pi x} }} \sim \sum\limits_{n \ge 0} {\frac{{a_n }}{{x^n }}} .
\end{equation}
\end{theorem}

\begin{proof} As $x \rightarrow \infty$ we have
\[
\cfrac{{\Gamma \left( {x + 1} \right)e^x }}{{x^x \sqrt {2\pi \left( {x + \frac{1}{6}} \right)} }} = \frac{{\Gamma \left( {x + 1} \right)e^x }}{{x^x \sqrt {2\pi x} }}\left( {1 + \frac{1}{{6x}}} \right)^{ - 1/2}  \sim \left( {1 + \frac{1}{{6x}}} \right)^{ - 1/2} \sum\limits_{n \ge 0} {\frac{{a_n }}{{x^n }}} .
\]
From the binomial formula we find
\[
\left( {1 + \frac{1}{{6x}}} \right)^{ - 1/2}  \sim \sum\limits_{n \ge 0} { \binom{- 1/2}{n}\frac{1}{{6^n x^n }}} ,
\]
as $x \rightarrow \infty$. Thus we obtain the asymptotic expansion
\begin{equation}\label{eq1}
\cfrac{{\Gamma \left( {x + 1} \right)e^x }}{{x^x \sqrt {2\pi \left( {x + \frac{1}{6}} \right)} }} \sim \sum\limits_{n \ge 0} {\left( {\sum\limits_{j = 0}^n { \binom{- 1/2}{j}\frac{{a_{n - j} }}{{6^j }}} } \right)\frac{1}{{x^n }}} .
\end{equation}
Suppose the expansion of the form
\[
\sum\limits_{n \ge 0} {\left( {\sum\limits_{j = 0}^n { \binom{- 1/2}{j}\frac{{a_{n - j} }}{{6^j }}} } \right)\frac{1}{{x^n }}}  \sim \sum\limits_{n \ge 0} {\frac{{g_n }}{{\left( {x + v_n } \right)^{2n} }}} .
\]
Now we expand the right hand side in powers of $1/x$:
\begin{align*}
\sum\limits_{n \ge 0} {\frac{{g_n }}{{\left( {x + v_n } \right)^{2n} }}} & = \sum\limits_{n \ge 0} {\frac{{g_n }}{{x^{2n} }}\left( {1 + \frac{{v_n }}{x}} \right)^{ - 2n} } = \sum\limits_{n \ge 0} {\frac{{g_n }}{{x^{2n} }}\sum\limits_{l \ge 0} { \binom{- 2n}{l}\frac{{v_n^l }}{{x^l }}} } \\
& = \sum\limits_{n \ge 0} {\sum\limits_{l \ge 0} { \binom{- 2n}{l}\frac{{g_n v_n^l }}{{x^{2n + l} }}} }  = \sum\limits_{n \ge 0} {\left( {\sum\limits_{j = 0}^{\left\lfloor {n/2} \right\rfloor } { \binom{- 2j}{n - 2j}g_j v_j^{n - 2j} } } \right)\frac{1}{{x^n }}} .
\end{align*}
According to the uniqueness of asymptotic series, the proof is complete.
\end{proof}
From \eqref{formula4} we find
\begin{align*}
& g_0  = 1,\;g_1  = \frac{1}{{144}},\;g_2  =  - \frac{{3857}}{{3110400}},\;g_3  = \frac{{20932906335329}}{{34283052002304000}}, \ldots ,\\
& v_1  = \frac{{23}}{{90}},\;v_2  = \frac{{1792627}}{{7289730}},\;v_3  = \frac{{570984637359867601981}}{{2288928529497568067550}}, \ldots \;.
\end{align*}
Note that $v_0$ vanishes in \eqref{expansion} due to the zero power. We have obtained an expansion in even powers however, the shift sequence $\left\{ {v_n } \right\}_{n \geq 0}$ has different terms whereas \eqref{formula3} has constant ($= 1/4$) shift in all terms. Numerical evaluation of the first few $v_n$ 
\begin{align*}
v_1 & = 0.2555555555 \ldots ,\\
v_2 & = 0.2459113026 \ldots ,\\
v_3 & = 0.2494549873 \ldots ,\\
v_4 & = 0.2498398924 \ldots ,\\
v_5 & = 0.2499584970 \ldots ,\\
v_6 & = 0.2499884146 \ldots ,\\
v_7 & = 0.2499963289 \ldots ,
\end{align*}
leads us to the
\begin{conjecture}
$\lim _{n \to \infty } v_n  = 1/4$.
\end{conjecture}
This conjecture suggests a new asymptotic series to the Gamma function in terms of $1/\left(x+1/4\right)$. Our second result is
\begin{theorem} The Gamma function has an asymptotic series expansion of the form
\begin{equation}\label{formula5}
\Gamma \left( {x + 1} \right) \sim x^x e^{ - x} \sqrt {2\pi \left( {x + \frac{1}{6}} \right)} \sum\limits_{k \ge 0} {\frac{{G_k }}{{\left( {x + \frac{1}{4}} \right)^k }}},
\end{equation}
as $x \rightarrow \infty$, where the the $G_k$ coefficients are given by
\begin{equation}\label{formula6}
G_k  = \sum\limits_{j = 0}^k {\binom{ - 1/2}{j} \frac{{a_{k - j} }}{{6^j }}}  - \sum\limits_{j = 0}^{k - 1} {\binom{-j}{k-j}\frac{{G_j }}{{4^{k - j} }}} .
\end{equation}
The $a_n$ coefficients again, are from \eqref{formula1}.
\end{theorem}

\begin{proof} Similary to the previous proof we expand our series in terms of $1/x$:
\begin{align*}
\sum\limits_{k \ge 0} {\frac{{G_k }}{{\left( {x + \frac{1}{4}} \right)^k }}}  = \sum\limits_{k \ge 0} {\frac{{G_k }}{{x^k }}} \sum\limits_{l \ge 0} {\binom{-k}{l}\frac{1}{{4^l }}\frac{1}{{x^l }}} & = \sum\limits_{k \ge 0} {\sum\limits_{l \ge 0} {\binom{-k }{ l}\frac{{G_k }}{{4^l }}\frac{1}{{x^{k + l} }}} }\\
& = \sum\limits_{k \ge 0} {\left\{ {\sum\limits_{j = 0}^k {\binom{-j }{ k-j}\frac{{G_j }}{{4^{k - j} }}} } \right\}\frac{1}{{x^k }}} .
\end{align*}
According to \eqref{eq1} and the uniqueness of asymptotic series the proof of \eqref{formula6} is complete.
\end{proof}

\section{Numerical comparisons}
We will compare in this paragraph the numerical performance of some asymptotic formulas to the Gamma function with our new formulas for large values. We compare the following approximation formulas for $\Gamma \left(x+1\right)$.

\begin{align}
 & \left( {\frac{x}{e}} \right)^x \sqrt {2\pi x} \exp \left( {\frac{1}{{12x}} - \frac{1}{{360x^3 }} + \frac{1}{{1260x^5 }} - \cdots} \right)\quad \left( {{\rm Stirling}} \right)  ,\\
 & \left( {\frac{x}{e}} \right)^x  \sqrt {2\pi x} \left( {1 + \frac{1}{{12x}} + \frac{1}{{288x^2 }} - \frac{139}{{51840x^3 }} - \cdots} \right)\quad \left( {{\rm Laplace}} \right),\\
 & \left( {\frac{x}{e}} \right)^x  \sqrt {2\pi x} \sqrt[6]{{\left( {1 + \frac{1}{{2x}} + \frac{1}{{8x^2 }} + \frac{1}{{240x^3 }} - \cdots} \right)}}\quad \left( {{\rm Ramanujan}} \right),\\
 & \left( {\frac{x}{e}} \right)^x \sqrt {2\pi x} \sqrt {\left( {1 + \frac{1}{{6x}} + \frac{1}{{72x^2 }} - \frac{{31}}{{6480x^3 }} -  \cdots } \right)}\quad \left( {{\rm Mortici}} \right),\\
 & \left( {\frac{x}{e}} \right)^x \sqrt {2\pi \left( {x + \frac{1}{6}} \right)} \left( {1 + \frac{1}{{144\left( {x + \frac{1}{4}} \right)^2 }} - \frac{1}{{12960\left( {x + \frac{1}{4}} \right)^3 }} -  \cdots } \right) \quad \left( {{\rm New}} \right)  .
\end{align}

The following table displays the number of exact decimal digits (edd) of the formulas for some values of $x$. Exact decimal digits are defined as follows:
\begin{equation}
{\rm edd}\left( x \right) =  - \log_{10} \left| {1 - \frac{{{\rm approximation}\left( x \right)}}{\Gamma \left(x+1\right)}} \right| .
\end{equation}
In the table below the $(i)$-th entry ($i=1,2,\ldots$) in a line starting with ``name'' is the edd of the given approximation using the series up to the $i$-th order term. The ``$-$'' sign indicates that the approximation is smaller and the ``$+$'' sign (not displayed) indicates that the approximation is larger than the true value. Note that in the case of Stirling's formula the first $n$ terms of the asymptotic series give the $2n$th order approximation.

\begin{table*}[ht]
\begin{center}
\begin{tabular}
[c]{ l r r r r r r r r r r r r }\hline\hline
Formula & $x$ & $(1)$ & $(2)$ & $(3)$ & $(4)$ & $(5)$ & $(6)$ & $(7)$ & $(8)$ \\ \hline\hline
 & & & & & &\\  
 Stirling & 100  &  & 8.6  &   & -13.1 &   & 17.2 &   & -21.1 \\
 Laplace & 100  & \textbf{-6.5} & 8.6 & \textbf{11.7} & -13.1 & \textbf{16.2} & 17.2 & \textbf{-20.4} & -21.1  \\
 Ramanujan & 100  & -5.7 & -9.2 & 11.0 & -13.3 & -15.4 & 17.3 & 19.5 & -21.1 \\
 Mortici & 100  & -6.2 & 8.6 & 11.4 & -13.1 & -15.9 & 17.2 & -20.0 & -21.1 \\
 New & 100  & -6.2 & \textbf{10.1} & 10.9 & \textbf{-14.9} & -15.2  & \textbf{19.4} & 19.2 & \textbf{-23.0} \\
 & & & & & &\\
 Stirling & 1000  &   & 11.6  &   & -18.1 &   & 24.2 &   & -30.1  \\
 Laplace & 1000  & \textbf{-8.5} & 11.6 & \textbf{15.6} & -18.1 & \textbf{22.2} & 24.2 & \textbf{-28.3} & -30.1   \\
 Ramanujan & 1000  & -7.7 & -12.2 & 15.0 & -18.3 & -21.4 & 24.3 & 27.5 & -30.1   \\
 Mortici & 1000  & -8.2 & 11.6 & 15.4 & -18.1 & -21.9 & 24.2 & -28.0 & -30.1   \\
 New & 1000  & -8.2 & \textbf{13.1} & 14.9 & \textbf{-19.7} & -21.2 & \textbf{26.9} & 27.2 & \textbf{-33.5} \\
 & & & & & &\\
 Stirling & 10000  &  & 14.6  &  & -23.1 &  & 31.2 &  & -39.1   \\
 Laplace & 10000  & \textbf{-10.5} & 14.6 & \textbf{19.6} & -23.1 & \textbf{28.2} & 31.2 & \textbf{-36.3} & -39.1   \\
 Ramanujan & 10000  & -9.7 & -15.2 & 19.0 & -23.3 & -27.4 & 31.3 & 35.5 & -39.1   \\
 Mortici & 10000  & -10.2 & 14.6 & 19.3 & -23.1 & -27.9 & 31.2 & -36.0 & -39.1   \\
 New & 10000  & -10.2 & \textbf{16.1} & 18.9 & \textbf{-24.7} & -27.2 & \textbf{33.7} & 35.2 & \textbf{-42.1} \\
\end{tabular}
\end{center}
\hfill{}
\caption{The number of exact decimal digits of the asymptotic series for some values of $x$.}
\end{table*}

\subsection*{Conclusion} It is seen that when we use odd order approximations, Laplace's formula is the most accurate. In the case of even orders Ramanujan's approximation is better than Stirling's, Laplace's and the one by Mortici, but our new formula gives better approximations even than that of Ramanujan's.

\bigskip

Expression \eqref{expansion} is a slightly different than the previous ones, thus we consider it's edds in a separate table. The notations are the same except the fact that this series contains only even order terms.

\begin{table*}[ht]
\begin{center}
\begin{tabular}
[c]{ l r r r r r r r r }\hline\hline
Formula & $x$ & $(2)$ & $(4)$ & $(6)$ & $(8)$ & $(10)$ \\\hline\hline
 & & & & & &\\
 Special & 100  & 10.9 & -15.2 & 19.2 & -22.9 & 26.5 \\
 Special & 1000  & 14.9 & -21.2 & 27.2 & -32.9 & 38.5 \\
 Special & 10000  & 18.9 & -27.2 & 35.2 & -42.9 & -50.5 \\
\end{tabular}
\end{center}
\hfill{}
\caption{The number of exact decimal digits of the special asymptotic series \eqref{expansion} for some values of $x$.}
\end{table*}

\appendix
\section{Tables of coefficients}\label{appendix}
The following table gives the first few values of the sequences $\left\{ {g_n } \right\}_{n \geq 0}$ and $\left\{ {v_n } \right\}_{n \geq 0}$ appearing in \eqref{expansion}.
\begin{table*}[ht]
\begin{center}
\begin{tabular}[c]{|l|r|l|r|} \hline
$$ & $$ & $$ & $$\\ [-1.5ex]
$g_{0}$ & {$1$} & $v_{0}$ & {}\\ [1.5ex]
$g_{1}$ & {$\frac{1}{144}$} & $v_{1}$ & {$\frac{23}{90}$}\\ [1.5ex]
$g_{2}$ & {$-\frac{3857}{3110400}$} & $v_{2}$ & {$\frac{1792627}{7289730}$}\\ [1.5ex]
$g_{3}$ & {$\frac{20932906335329}{34283052002304000}$} & $v_{3}$ & {$\frac{570984637359867601981}{2288928529497568067550}$}\\ [1.5ex]
\hline
\end{tabular}
\end{center}
\hfill{}
\caption{The first few exact values of the sequences $\left\{ {g_n } \right\}_{n \geq 0}$ and $\left\{ {v_n } \right\}_{n \geq 0}$.}
\end{table*}

The table below gives the first few numerical values of the sequences $\left\{ {g_n } \right\}_{n \geq 0}$ and $\left\{ {v_n } \right\}_{n \geq 0}$.
\begin{table*}[ht]
\begin{center}
\begin{tabular}[c]{|l|r|l|r|} \hline
$g_{0}$ & {$1.000000000000000$} & $v_{0}$ & {}\\ [0.1ex]
$g_{1}$ & {$0.006944444444444$} & $v_{1}$ & {$0.255555555555555$}\\ [0.1ex]
$g_{2}$ & {$-0.001240033436214$} & $v_{2}$ & {$0.245911302613402$}\\ [0.1ex]
$g_{3}$ & {$0.000610590513759$} & $v_{3}$ & {$0.249454987345193$}\\ [0.1ex]
$g_{4}$ & {$-0.000655407405149$} & $v_{4}$ & {$0.249839892410196$}\\ [0.1ex]
$g_{5}$ & {$0.001199164540953$} & $v_{5}$ & {$0.249958497082160$}\\ [0.1ex]
\hline
\end{tabular}
\end{center}
\hfill{}
\caption{The first few numerical values of the sequences $\left\{ {g_n } \right\}_{n \geq 0}$ and $\left\{ {v_n } \right\}_{n \geq 0}$.}
\end{table*}

The following table gives the first few coefficients appearing in the asymptotic series \eqref{formula5}.
\begin{table*}[ht]
\begin{center}
\begin{tabular}[c]{|l|r|l|r|l|r|} \hline
$$ & $$ & $$ & $$ & $$ & $$\\ [-1.5ex]
$G_{0}$ & {$1$} & $G_{5}$ & {$-\frac{53}{2612736}$} & $G_{10}$ & {$\frac{360182239526821}{300361133850624000}$}\\ [1.5ex]
$G_{1}$ & {$0$} & $G_{6}$ & {$\frac{5741173}{9405849600}$} & $G_{11}$ & {$\frac{104939254406053}{210853515963138048000}$}\\ [1.5ex]
$G_{2}$ & {$\frac{1}{144}$} & $G_{7}$ & {$\frac{37529}{18811699200}$} & $G_{12}$ & {$-\frac{508096766056991140541}{151814531493459394560000}$}\\ [1.5ex]
$G_{3}$ & {$-\frac{1}{12960}$} & $G_{8}$ & {\large$-$\Large$\frac{710165119}{1083553873920}$} & $G_{13}$ & {$-\frac{70637580369737593}{151814531493459394560000}$}\\ [1.5ex]
$G_{4}$ & {$-\frac{257}{207360}$} & $G_{9}$ & {$-\frac{3376971533}{4022693756928000}$} & $G_{14}$ & {$\frac{289375690552473442964467}{21861292535058152816640000}$}\\ [1.5ex]
\hline
\end{tabular}
\end{center}
\hfill{}
\caption{The first few $G_{k}$ coefficients.}
\end{table*}

\bibliographystyle{amsplain}

\end{document}